\documentclass[11pt, a4paper]{amsart}

\usepackage{geometry} 
\usepackage{fancyhdr}
\usepackage{url}
\usepackage{amssymb}
\usepackage{amsmath}
\usepackage{accents}
\usepackage{eucal}
\usepackage{amscd}
\usepackage[shortalphabetic]{amsrefs}
\usepackage{times}

\allowdisplaybreaks[1]

\newtheorem{thm}{Theorem}[section]
\newtheorem{lem}[thm]{Lemma}
\newtheorem{prop}[thm]{Proposition}

\numberwithin{equation}{section}

\providecommand{\abs}[1]{\lvert#1\rvert}

\newcommand{\ho}{\accentset{\circ}{h}}

\newcommand{\p}{\partial}
\newcommand{\nablap}{\accentset{\bot}{\nabla}}
\newcommand{\Rp}{\accentset{\perp}{R}}

\author{Charles Baker}\thanks{Research partially supported by Discovery Grant DP0985802 from the Australian Research Council and partially by The Leverhulme Trust}
\email{rogercharlesbaker@gmail.com}
\title[Singularities of the mean curvature flow in high codimension]{A partial classification of type I singularities of the mean curvature flow in high codimension}
\begin{document}
\maketitle

\section{Introduction}
The purpose of this article is to examine the possible shapes of type I singularities that form in the mean curvature flow of submanifolds of arbitrary codimension, assuming that the initial submanifold satisfies a particular curvature pinching condition.  The mean curvature flow of an initial immersion $F_0 : M^n \rightarrow \mathbb{R}^{n+k}$, where we always assume $M^n$ is a closed manifold, is given by a time-dependent family of immersions $F : M \times [0,T) \rightarrow \mathbb{R}^{n+k}$ that satisfy
\begin{equation*}
	\begin{cases}
		\frac{\p}{\p t} F(p,t) = H(p,t), \quad p \in M, \, t \geq 0 \\
		F( \cdot, 0) = F_0.
	\end{cases}
\end{equation*}
The author and Ben Andrews (\cite{AB}) recently proved the following Pinching Lemma for submanifolds of arbitrary codimension moving by the mean curvature flow:
\begin{lem}\label{l: preservation of pinching}
If a solution $F : M^n \times[0,T) \to \mathbb{R}^{n+k}$ of the mean curvature flow satisfies $\abs{h}^2 +a < c\abs{H}^2$ for some constants $ c \leq \frac1n + \frac{1}{3n}$ and $a>0$ at $t = 0$, then this remains true for all $0 \leq t < T$.
\end{lem}
The main theorems we present in this article use this Pinching Lemma in combination with blow-up arguments and other techniques introduced by Huisken for hypersurfaces to characterise the asymptotic shape of submanifolds evolving by the mean curvature flow as the first singular time is approached.  The main classification theorem we obtain is the following:
\begin{thm}\label{thm: general class}
Suppose that $F_{\infty} : M_{\infty}^n \times (-\infty, 0) \rightarrow \mathbb{R}^{n+k}$ arises as a blow-up limit of the mean curvature flow $F: M^n \times [0, T) \rightarrow \mathbb{R}^{n+k}$ about a special singular point.  If $F_0(M)$ satisfies $\abs{ H }_{ \text{min} } > 0$ and $\abs{ h }^2 \leq 4/(3n) \abs{ H }^2$, then $F_{\infty}(M_{\infty},-1/2)$ must be a sphere $\mathbb{S}^n(\sqrt{n})$ or one of the cylinders $\mathbb{S}^{m}(\sqrt{m}) \times \mathbb{R}^{n-m}$, where $1 \leq m \leq n-1$.
\end{thm}

In \cite{AB} we show that if a submanifold satisfies a suitable pinching condition, then the mean curvature flow evolves the submanifold to round point in finite time. In this article we relax the pinching of the initial submanifold and seek to understand the asymptotic shape of the evolving submanifold as we approach the first singular time.  We still assume that $\abs{ H }$ is everywhere positive initially, however we no longer expect the entire submanifold to disappear at the maximal time.  In the case of mean-convex hypersurfaces, a classification of type I singularities was achieved by Huisken in \cite{Hu2} and \cite{Hu3}.  A key ingredient in this analysis was Huisken's monontoncity formula, introduced in \cite{Hu2}, which also holds in arbitrary codimension.  The singularities classified by Husiken in \cite{Hu2} are a special kind of type I singularity called a `special' type I singularity.  The more general kind of singularity is naturally called a `general' type I singularity and in order to have a complete understanding of type I singularity formation, we must also able to treat general singularities (definitions of the various kinds of singularities follow).  In the case of embedded hypersurfaces, the classification of general type I singularities is due to Stone \cite{St}.

Smoczyk (\cite{Sm}) has previously classified type I blow-up limits of the the mean curvature flow in high codimension under the assumption that the blow-up limits have flat normal bundle.  Although this curvature condition is in general much more restrictive than the pinching condition we work with, Smoczyk's classification includes products of Euclidean space with an Abresch-Langer curve, which also appear in the hypersurface classification, however they do not appear in our classification.  These spaces do not feature in our classification as they do not satisfy our pinching condition $\abs{ h }^2 \leq 4/(3n) \abs{H}^2$.  It's worthwhile to point out that the condition of having flat normal bundle is not preserved by the mean curvature flow.

This article is organised as follows.  We begin by recalling some basic facts about the mean curvature flow, in particular the notion of a self-similar solution.  In the third section we detail the blow-up argument we use in order to obtain the limit flow.  In Huisken's original papers he uses a continuous rescaling process to obtain a limit hypersurface, whereas we consider a sequence of rescaled flows and produce a limit flow.  We also give a new proof of Langer's compactness theorem for immersed submanifolds of arbitrary codimension using the using the well-known Cheeger-Gromov compactness theorem, and then use it prove a compactness theorem for mean curvature flows.  As will be immediately obvious, we have been heavily influenced by Hamilton's treatment of singularities in the Ricci flow, and in particular his compactness theorem for Ricci flows (\cite{Ham}).  In the latter sections we use Huisken's monotonicity formula and our compactness theorem for mean curvature flows to characterise the possible singular behaviour of the mean curvature flow, assuming the initial submanifold satisfies the curvature pinching condition of the Pinching Lemma.  An important detail of our classification theorem is the difference between special and general type I singularities, notions introduced by Andrew Stone in \cite{St}.  As embeddedness is not in general preserved in high codimension, we only obtain a classification of special type I singularities (achieved for hypersurfaces in \cite{Hu2} and \cite{Hu3}), and not of general type I singularities, which was successfully achieved by Stone for embedded hypersurfaces in \cite{St}.  In the last section, we use our compactness theorem for mean curvature flows to simplify the proof of the limiting spherical shape of the evolving submanifolds considered in \cite{Hu1} and \cite{AB}, and we also present a new estimate which simplifies the convergence arguments of the normalised flow given in \cite{Hu1}.

\section{Some basic facts about the mean curvature flow}
In this section we recall some some basic facts about the mean curvature flow. These results are well-known to researchers working in the field, however we give complete proofs to highlight why they are still true, or why they fail, in high codimension.

A self-similar solution of the mean curvature flow is defined to be one whose image moves only by scaling, that is a solution which moves by
\begin{equation}
	F(M, t) =  \lambda(t) F(M, t_1) \label{eqn: defn MCF soliton1},
\end{equation}
where $\lambda(t)$ is a time-dependent function and $t_1 \in [0, T)$ is some fixed time.  It is well-known that the image of the mean curvature flow is invariant under tangential reparametrisations, so more precisely we define a self-similar solution to be one that satisifes
\begin{equation}
	F(p,t) = \lambda(t) F( \phi(p,t), t_1) \label{eqn: defn MCF soliton2}
\end{equation}
where $\lambda(t)$ and $t_1$ are as before and $\phi(p,t) : M^n \times [0, T) \rightarrow M^n$ is a time-dependent family of diffeomorphisms.

We would like an equivalent characterisation of self-similar solutions in terms of the mean curvature vector.  By differentiating \eqref{eqn: defn MCF soliton2}
we obtain
\begin{equation}
	\frac{ \p }{ \p t }F(p,t) = \lambda'(t) F(\phi(p,t), t_1) + \lambda(t)\nabla_k F(\phi(p,t), t_1) \frac{d\phi(p,t)}{dt}. \label{eqn: ss1}
\end{equation}
The normal and tangential components are given by
\begin{align*}
	&\left( \frac{ \p }{ \p t }F(p,t) \right)^{\perp} = \lambda'(t) F^{\perp}(\phi(p,t), t_1) \\
	&\left( \frac{ \p }{ \p t}F(p,t) \right)^{\top} = \lambda'(t) F^{\top}(\phi(p,t), t_1) + \lambda(t)\nabla_k F(\phi(p,t), t_1) \frac{d\phi^k(p,t)}{dt}.
\end{align*}
The immersion $F_t(M)$ satisfies the mean curvature flow, so the tangential component must be equal to zero.  This prescribes an ODE for $\phi(t)$, and since $M$ is closed, standard ODE techniques guarantee a solution exists.  Since $F_t(M)$ solves the mean curvature flow, the normal component satisfies
\begin{equation}
	H(p,t) = \lambda'(t) F^{\perp}(\phi(p,t), t_1). \label{eqn: ss2}
\end{equation}
The mean curvature scales like $H(p,t) = H(\phi(p,t),t_1)/\lambda(t)$, then substituting this into \eqref{eqn: ss2} we get
\begin{equation*}
	\lambda(t)\lambda'(t) = \frac{ H(\phi(p,t), t_1) }{ F^{\perp}(\phi(p,t), t_1) }.
\end{equation*}
This can be re-written as
\[ \frac{ d }{ dt } \lambda(t)^2 = \frac{ 2H(\phi(p,t), t_1) }{ F^{\perp}(\phi(p,t), t_1) }, \]
then integrating with the initial condition $\lambda(t_1) = 1$ we obtain
\begin{equation}
	 \lambda(t) = \sqrt{ 1 + \int_{t_1}^t \frac{ 2H }{ F^{\perp} }(\phi(p,\tau), t_1) \, d\tau }. \label{eqn: ss3}
\end{equation}
The left hand side is independent of the point $p$.  Thus, for any points $p_1, p_2 \in M$ we have
\[ \int_{t_1}^t \frac{ 2H }{ F^{\perp} }(\phi(p_1,\tau), t_1) = \int_{t_1}^t \frac{ 2H }{ F^{\perp} }(\phi(p_2,\tau), t_1). \]
After differentiating with respect to $t$ we deduce that
\[ \frac{ H(\phi(p,t), t_1) }{ F^{\perp}(\phi(p,t), t_1) } = \frac{ \alpha }{ 2 } \]
for all $p \in M$, where $\alpha$ is a constant (we choose $\alpha/2$) to agree with the presentation in \cite{Eck}).  Substituting this back into \eqref{eqn: ss3} we get
\begin{equation*}
	 \lambda(t) = \sqrt{ 1 + \alpha(t - t_1) }
\end{equation*}
for all $t$ satisfying $1 + \alpha(t - t_1) \geq 0$.  Differentiating this formula we get $\lambda'(t) = -\alpha/(2\lambda(t))$, then combining this with \eqref{eqn: ss2} and \eqref{eqn: defn MCF soliton1} we find
\begin{equation}
	H(p,t) = \frac{ \alpha }{ 2 \lambda^2(t) } F^{\perp}(p,t)
\end{equation}
This equation describes shrinking solitons for $\alpha < 0$ and expanding solitons for $\alpha > 0$.  Here we consider the case $\alpha < 0$.  To make contact with our blow-up construction that appears later, we assume that the solution is defined for $t \in (-\infty, 0)$ and that it shrinks to a point at $t=0$.  This leads us to set $\alpha = 1/t_1$ (this is negative because $t$ is a backwards time scale), then $\lambda(t) = \sqrt{t/t_1}$ and equation \eqref{eqn: ss2} becomes
\begin{equation}
	H(p,t) = \frac{1}{2t}F^{\perp}(p,t). \label{eqn: MCF soliton}
\end{equation}
In particular, if we take $t_1 = -1$, then equation \eqref{eqn: defn MCF soliton1} is simply
\[ F_t(M) = \sqrt{-t}F_{-1}(M), \]
and at $t=-1$ equation \eqref{eqn: MCF soliton} is
\[ H(p,t) = \frac{-1}{2}F^{\perp}(p,t). \]
We have just shown that a self-similar solution of the mean curvature flow satisifies equation \eqref{eqn: MCF soliton}.  The converse is also easily shown to be true.  To this end, assume that a solution $F$ of the mean curvature flow satisfies equation \eqref{eqn: MCF soliton}.  For the moment, assume that there exists a time dependent family of diffeomorphisms $\varphi_t : M \times [0,T) \rightarrow M$.  We wish to show that
\begin{equation}\label{eqn: ss converse}
	\frac{ \p }{ \p t } \left( \frac{ F( \phi(p,t), t ) }{ \sqrt{ -t } } \right) = 0.
\end{equation}
We compute
\begin{equation*}
	\frac{ \p }{ \p t } \left( \frac{ F( \phi(p,t), t ) }{ \sqrt{ -t } } \right) =  \frac{ 1 }{ \sqrt{-t} } \left(  \frac{ \p F }{ \p t }( \phi(p,t), t ) + \nabla_k F( \phi(p,t), t ) \frac{ d\phi^k(p) }{ dt } \right) + \frac{ F( \phi(p,t),t ) }{ (-s)^{3/2} }.
\end{equation*}
Resolving this equation into tangential and normal components and equating them both to zero, we find the tangential components satisfy the ODE
\[ \left( \frac{ \p }{ \p t } F( \phi(p,t), t ) \right)^{ \top } + \nabla_k F( \phi(p,t), t ) \frac{ d\phi^k(p) }{ dt }  + \frac{ F( \phi(p,t),t ) }{ 2t } = 0,
\]
which is again solvable by standard ODE techniques.  Since, by assumption, $F$ satisfies
\[ H( \phi(p,t),t) = \frac{1}{2t}F^{\perp}( \phi(p,t),t) ), \]
the normal components also sum to zero, thus confirming equation \eqref{eqn: ss converse}.  Integrating \eqref{eqn: ss converse} with $F_{-1}(M)$ as the initial condition gives
\[ F_t(M) = \sqrt{-t} F_{-1}(M) \]
as desired.  We collect the above into the following proposition:
\begin{prop}
Let $F : M^n \times (-\infty, 0) \rightarrow \mathbb{R}^{n+k}$ be a solution of the mean curvature flow.  Then $F_t(M^n)$ is self-similar, that is $F_t(M) = \sqrt{-t}F_{-1}(M)$, if and only if $F_t(M)$ satisfies
\[ H(p,t) = \frac{1}{2t}F^{\perp} (p,t) \]
for all $t \in (-\infty, 0)$.
\end{prop}

We are not able to make Stone's argument in \cite{St} work in high codimension, essentially because embeddedness is not in general preserved in high codimension, and therefore the limit flow could have multiplicity greater than one.  Let us see why embeddedness of hypersurfaces is preserved, and why the proof breaks down in higher codimension.

\begin{prop}\label{prop: preservation of embeddedness}
Suppose $F_t(M^n)$ is a hypersurface moving by the mean curvature flow in Euclidean space.  If $F_0(M)$ is embedded, then $F_t(M)$ remains embedded for as long as the flow is defined.
\end{prop}
\begin{proof}
The mean curvature flow of submanifolds of any codimension, and in particular hypersurfaces, has a solution for as long as the second fundamental form remains bounded.  Hence, we must show that if a solution to the mean curvature flow exists on the closed interval $[0, T]$ (and therefore $\max_{p \in M} \abs{ h }^2(p,t) \leq C$ for all $t \in [0,T]$ by what we have just said) and it is initially embedded, then the solution remains embedded on the same closed time interval.  We work with the extrinsic distance function
\[ d : M \times M \times [0,T] \to \mathbb{R}. \]
If $M_0$ is embedded, then clearly $d >0$.  The idea of the proof is to use the maximum principle to show that $d$ is non-decreasing along the flow.  The small problem with doing this is that $d(p,p) =0$.  To overcome this problem we first remove a small open strip $S$ about the diagonal.  We shall see the curvature bound can be used to show that $d \geq 0$ in $S$ with equality only on the diagonal, and then the maximum principle can be used to show that $d$ is non-decreasing on $(M \times M \times [0,T]) \setminus S$.

Let $p, q \in M$ and $\gamma(s)$ be an arc-length parametrised length minimising geodesic connecting $p$ to $q$.  Denote by $l$ the intrinsic length of $\gamma$ and write $\dot \gamma$ for $d\gamma / d s$.  From the curvature bound we have
\[ \abs{ \ddot \gamma(s) } = \abs{ h(\dot \gamma, \dot \gamma) } \leq C. \]
Since $\gamma$ is unit speed, $\abs{ \dot \gamma(s) }$ =1, so by the Cauchy-Schwarz inequality
\[ \left| \frac{ d }{ ds } \langle \dot \gamma(s), \dot \gamma(0) \rangle \right| \leq \abs{ \ddot \gamma(s) } \abs{ \gamma(0) } \leq C. \]
We compute
\begin{align*}
	\abs{ \langle \dot \gamma(s), \dot \gamma(0) \rangle - \langle \dot \gamma(0), \dot \gamma(0) \rangle } &= \left| \in0^s \frac{ d }{ d\sigma } \langle \dot \gamma(\sigma), \dot \gamma(0) \rangle \, d\sigma \right| \\
	&\leq Cs \leq Cl,
\end{align*}
so if $l \leq 1/(2C)$, then it follows that
\[ \langle \dot \gamma(s), \dot \gamma(0) \rangle \geq 1/2, \]
where we have used $\langle \dot \gamma(0), \dot \gamma(0) \rangle =1$.  Another use of the Cauchy-Schwarz inequality shows
\begin{align*}
	d &= \abs{ q - p } \geq \abs{ \langle q-p \rangle, \dot \gamma(0) } \\
	&= \left| \int_{ \gamma } \langle \dot \gamma(s), \dot \gamma(0) \rangle ds \right| \geq \frac12.
\end{align*}
Thus on the strip
\[ S = \{ (p,q,t) \in M \times M \times [0,T] : l(p,q,t) < \frac1{2C} \}, \]
$d \geq 0$ with equality if and only if $p=q$.  The parabolic boundary of the complementary domain $U = (M \times M \times [0,T])$ is
\[ \p_p U = \p U_1 \cup \p U_2, \]
where
\[ \p U_1= \{ (p,q,t) \in M \times M \times (0,T) : l(p,q,t) = \frac1{2C} \} \]
and
\[ \p U_2 = \{ (p,q,t) \in M \times M \times \{ 0 \} : l(p,q,t) \geq \frac1{2C} \}. \]
Since $F_0(M)$ is embedded and $M$ is closed, $d \geq d_0 >0$ initially.  On the set  $\p U_1$, $d = 1/(2C)$  by the above construction of the set $U$, and on $\p U_2$, $d \geq d_0$.  We now show that $d$ is non-decreasing on $U$, which completes the proof.  We point out that up to this point, the argument works for submanifolds of any codimension.

To simplify the computations, we work with the square of the distance function.  A minimum of $d$ is obviously also a minimum of $d^2$.  To apply the maximum principle, we first need to compute the first, second and time derivates of $d^2$ in some choice of local coordinates $\{ p^i \}$ near $p$ and $\{ q^i \}$ near $q$.  The first derivatives are
\begin{gather*}
	\p_{ q^j } d^2 = 2 \langle F(q) - F(p), \p_{q^j}F \rangle \\
	\p_{ p^j } d^2 = -2 \langle F(q) - F(p), \p_{p^j}F \rangle,
\end{gather*}
the second derivatives
\begin{gather*}
	\p_{ q^i } \p_{ q^j } d^2 = 2 g_{ij}^q - 2 \langle F(q) - F(p), h_{ij}^q \nu_q \rangle \\
	\p_{ p^i } \p_{ p^j } d^2 = 2 g_{ij}^p + 2 \langle F(q) - F(p),  h_{ij}^p \nu_p \rangle  \\
	\p_{ p^i } \p_{ q^j } d^2 = -2 \langle \p_{p^i}F, \p_{q^j}F \rangle,
\end{gather*}
and time derivative is
\begin{equation*}
	\p_t d^2 = 2 \langle F(q) - F(p), - H_q \nu_q + H_p \nu_p \rangle.
\end{equation*}
Here is the crucial point:  For a hypersurface, at a minimum of the distance function the tangent planes at $p$ and $q$ are parallel to each other.  We may therefore choose local coordinates such that  $\{ p^i \}$ and $\{ q^i \}$ are parallel for each $i$.  We now compute
 \begin{align*}
 	\frac{ \p d^2 }{ \p t } &- \left( g^{ij}_p \frac{ \p^2 d^2 }{ \p p^i \p p^j } + g^{ij}_q \frac{ \p^2 d^2 }{ \p q^i \p q^j } + 2g_p^{ik} g_q^{jl} \langle \p_{ p^k }F, \p_{ q^l }F \rangle\frac{ \p^2 d^2 }{ \p p^i \p q^j } \right) \\
	&= 2 \langle F(q) - F(p), -H_q \nu_q + H_p \nu_p \rangle - 2n - 2 \langle F(q) - F(p), H_p \nu_p \rangle - 2n \\
	&\quad + 2 \langle F(q) - F(p), H_q \nu_q \rangle + 4n \\
	&=0.
\end{align*}
We conclude by the maximum principle that $d^2$ is non-decreasing in time, which also completes the proof of the proposition.
\end{proof}
In the above proof it was crucial that we were able to choose parallel local coordiantes at the points $p$ and $q$ : without the good $4n$ contribution from the cross-derivative terms the proof does not work.  In high codimension this is not possible to do in general since the tangent planes could easily be orthogonal to each other at a point of minimum distance, which results in zero contribution from the cross-terms.

\begin{prop}
Let $F : M^n \rightarrow \mathbb{R}^{n+k}$ be an immersion of a complete, not necessarily compact, manifold $M$ and equip $M$ with the induced metric.  If the immersion satisfies $\sup_{p \in M} \abs{ h }^2 \leq C$, where $C$ is a uniform constant, then the injectivity radius of $M$ is uniformly bounded below, namely $inj_g \geq \delta >0$.
\end{prop}
\begin{proof}
From Klingerberg's Lemma (see, for example, \cite{Pet}), in order to control the injectivity radius from below, we need to have control on the maximum of the intrinsic sectional curvature and the length of the smallest closed geodesic loop.  By the Gauss relation, an upper bound on the second fundamental form gives an upper bound on the intrinsic sectional curvature of $M$.  The curvature of a small geodesic loop is simply an entry in the second fundamental form, so the upper bound on the second fundamental form immediately implies a lower bound on the length of the smallest possible geodesic loop.
\end{proof}
We see that control of the injectivity radius in the mean curvature flow is essentially for free.

\section{The blow-up construction and a compactness theorem for immersed submanifolds}
We remind the reader that $M$ is a fixed manifold, and that $M_t := F_t(M)$ refers to the immersed submanifold.   For a function $f \in C^{\infty}(\mathbb{R}^{n+k} \times \mathbb{R})$ defined on the ambient space, we follow standard abuse of notation and write
\[ \int_{M} f(F(p)) \, d\mu_g(p) = \int_{M} f(p)\, d\mu_g(p). \]
Integration over the manifold $M$ with respect to $d\mu_g$ and integration over the image $F(M)$ in $\mathbb{R}^{n+k}$ are linked by the area formula.  We denote the pushforward measure by $\mu = F(\mu_g)$, where for $U \in \mathbb{R}^{n+k}$ an open set, $\mu(U) := \mu_g(F^{-1}(U))$.  In order for the pushforward measure $\mu$ to be a Radon measure, the immersion $F$ must be a proper immersion.  Recall an immersion is proper if the inverse image of a compact set is also compact.  The area formula relates the induced measure on $M$ to the Hausdorff measure on $\mathbb{R}^{n+k}$ restricted to the image $F(M)$ of the immersion.  We denote Hausdorff measure on the ambient space by $d\mathcal{H}^{n+k}$ or simply by $d\mathcal{H}$.  For a $\mu_g$-measurable function $f : M \rightarrow \mathbb{R}$, by the area formula we have
\[ \int_{M} f(p) \, d\mu_g(p) = \int_{\mathbb{R}^{n+k} } \left( \sum_{p \in F^{-1}\{x\} } f(p) \right) \, d\mathcal{H}^{n+k}(x). \]
Choosing $f = \chi_{[F^{-1}(F(M))]}$ gives
\[ \int_{M} d\mu_g(p) = \int_{F(M)} \left( \sum_{p \in F^{-1}\{x\} } \right) \, d\mathcal{H}^{n+k}(x). \]
Thus, denoting the multiplicity function by $\theta$, we have $\mu = \theta \mathcal{H} \llcorner F(M)$.  In particular, if $F : M^n \rightarrow \mathbb{R}^{n+k}$ is a properly embedded submanifold, then $\theta \equiv 1$ and
\[ \int_{M} d\mu_g(p) = \int_{F(M)} d\mathcal{H}^{n+k}(x). \]

 For more details on Hausdorff measure and the area formula we refer the reader to \cite{EG}.  One final piece of notation before getting underway, for a point $p \in M$, we put $\lim_{t\rightarrow T} F(p) := \hat p \in \mathbb{R}^{n+k}$.

In order to study the asymptotic shape of the evolving submanifold $F_t(M)$ around a singular point as the first singular time is approached, we progressively `magnify' the solution around this point by considering a sequence of rescaled flows.  The limit of such rescaled flows is called a blow-up limit.  Our first task is to show how to obtain such a limit.  In order to obtain a smooth blow-up, we assume that the submanifold is developing a so-called type I singularity.  This imposes a natural maximum rate at which the singularity can develop, which then enables us to rescale at a rate that keeps the maximum curvature of the rescaled solution bounded.

Let $F : M^n \times [0, T) \rightarrow \mathbb{R}^{n+k}$ be a submanifold evolving by the mean curvature flow.  We say the submanifold is developing a type I singularity at time $T$ if there exists a constant $C_0 \geq 1$ such that
\begin{equation*}
	\max_{p \in M} \abs{ h }^2(p,t) \leq \frac{ C_0 }{ 2(T - t) }.
\end{equation*}
If no such constant exists, that is 
\[ \limsup_{t \rightarrow T} \max_{p \in M} \abs{ h }^2(p,t) (T-t), \]
then we say the submanifold is developing a type II singularity as $t \rightarrow T$.  We shall mainly be concerned with type I singularity formation, although we will work with type II singularities in the final section.  The blow-up rate of any singularity also satisfies the lower bound
\begin{equation*}
	\max_{ p \in M } \abs{ h }^2(p,t) \geq \frac{ 1 }{ 2(T-t) }
\end{equation*}
(see \cite{Hu2} or \cite{Man}), so a type I singularity satifies
\begin{equation*}
	\frac{1}{2(T-t)} \leq \max_{p \in M} \abs{ h }^2(p,t) \leq \frac{ C_0 }{ 2(T - t) }.
\end{equation*}

Let $q \in M$ be a fixed point and assume that the type I condition holds.  We want to rescale the solution around the point $\hat{q} \in \mathbb{R}^{n+k}$ by remaining time.  Let $(t_k)_{ k \in \mathbb{N} }$ be any sequence of times such that $t_k \rightarrow T$ as $k \rightarrow \infty$.  For example, we could take $t_k := T - 1/k$.  To rescale by remaining time we define set scale $\lambda_k := 1/\sqrt{2(T - t_k)}$.  We then define a sequence of parabolically rescaled flows
\begin{equation}\label{eqn: rescaled flows}
	F_k(p, s):= \lambda_k \big( F(p, T + s/\lambda_k^2)  - \hat q \big).
\end{equation}
Then for each $k$, $F_k : M \times [-\lambda_k^2T, 0)$ is a solution to the mean curvature flow (in the time variable $s$) that exists on the time interval $s \in [-\lambda_k^2 T, 0)$.  Under our parabolic rescaling the second fundamental form rescales like $\abs{ h }^2_k = \abs{ h }^2/ \lambda_k^2$, so using the type I hypothesis
\begin{align*}
	\abs{ h }^2_k(p,s) &= \frac{ \abs{ h }^2(p, T + s/\lambda_k^2) }{ \lambda_k^2 } \\
	&\leq 2(T - t_k) \cdot \frac{ C_0 }{ 2(T - T - s/\lambda_k^2) } \\
	&= \frac{ - C_0 }{ 2s }
\end{align*}
which holds on $s \in [-\lambda_k^2 T, 0)$.  Consequently, on the time intervals $I_k := (-\lambda_k^2T, 1/k)$, the rescaled flows have bounded second fundamental form.  Next we would like to use a compactness theorem for immersed submanifolds in order to obtain a limit flow.  The compactness theorem usually quoted in this context is \cite{Lan}.  The result presented in \cite{Lan} is for a sequence of two-surfaces of Euclidean three-space with $L^p$-bounded second fundamental form and a global area bound, whereas we need to apply the result to a sequence of $n$-dimensional submanifolds of codimension $k$ in the presence of bounds on all higher derivatives of the second fundamental form and only a local area bound.  Very recently we learned that in his PhD thesis Patrick Breuning has extended Langer's result to submanifolds of arbitrary codimension in the presence of a local area bound \cite{Br}.  We record Breuning's compactness theorem as follows:
\begin{thm}[Breuning-Langer compactness theorem for immersed submanifolds]
Let  $F_k : M_k^n \rightarrow \mathbb{R}^N$ be a sequence of proper immersions, where $M_k$ is a $n$-manifold without boundary and $p \in F_k(M_k)$.  Assume the following conditions are satisfied:
\begin{enumerate}
 	\item Uniform curvature derivative bounds: \\
	For each $k \in \mathbb{N}$,  for every $m \in \mathbb{N}$ there exists a constant $C_m(R)$ depending on $m$ and $R$ such that $\abs{ \nabla_k^m h_k}_{ F_k } \leq C_m$.
	\item Local area bound: \\
	For every $R > 0$ there exists a constant $C_R$ depending on R such that $\mu^k(B_R) \leq C_R$.
\end{enumerate}
Then there exists a proper immersion $F_{\infty} : M_{\infty} \rightarrow \mathbb{R}^{n+k}$, where $M_{\infty}$ is again a $n$-manifold without boundary, such that after passing to a subsequence there exists a sequence of diffeomorphisms $\phi_k : U_k \rightarrow (F_k)^{-1}(B_k) \subset M_k$, where $U_k \subset M_{\infty}$ are open sets with $U_k \subset\subset U_{k+1}$ and $M_{\infty} = \bigcup_{j=1}^{\infty} U_j$ such that $\phi_k^*F_k |_{U_j}$ converges in $C^{\infty}( U_j, \mathbb{R}^N)$ to $F_{\infty} |_{U_j}$.
\end{thm}
This is the essentially the statement of the Breuning's theorem in his thesis; we have simply changed some notation to conform with our own.  Note Breuning states the local area bound in terms of the pushforward measure.  Before we learned of Breuning's compactness theorem we did not know whether Langer's theorem did in fact hold in arbitrary codimension and we produced the following compactness theorem for immersed submanifolds in arbitrary codimension using the well-known compactness theorem of Cheeger and Gromov for abstract manifolds.  As an introduction to Cheeger-Gromov convergence and its application in Ricci flow we refer the reader to \cite{HA}.  For a complete proof of the Cheeger-Gromov compactness theorem itself, we refer the reader to \cite{Ham} as well as \cite{Pet}.  Our reference has been \cite{HA} and we use the definitions of smooth pointed Cheeger-Gromov convergence, the Arzel\`{a}-Ascoli theorem etc as they are presented in \cite{HA}.

Before we prove this theorem, we mention two issues that need to be dealt with in the proof.  First, the limit produced by applying the Cheeger-Gromov compactness theorem is an abstract limit that a priori loses all knowledge of the background space.  This creates the problem of ensuring that the limit $F_{\infty}$ is indeed an immersion, and the limit metric $g_{\infty}$ is the induced metric of the limit immersion.  A second problem that requires more work to deal with is the following:  We shall be considering pointed sequences of immersions, so the limit metric produced by the Cheeger-Gromov compactness theorem only sees the connected component of which it is part.  As an example, consider a sequence of tori where the tori simply lengthen off to infinity, and choose a base point on one side of the torus.  The Cheeger-Gromov limit of such a sequence is one cylinder, where the other side of the torus of which the base point was not part disappears in the limit.  In our extrinsic setting, we would like the limit to be two disconnected cylinders.

In order to deal with this situation, we introduce the following notion of convergence for a sequence of immersed submanifolds: For each $k \in \mathbb{N}$, let $M_k$ be a complete smooth manifold, $F_k : M_k \rightarrow \mathbb{R}^{N}$ a smooth immersion and $p_k \in M_k$ a basepoint. We say that $(M_k , F_k )$ converges to $(M_{\infty}, F_{\infty})$ on compact sets of $\mathbb{R}^N \times \mathbb{R}$ if there exists an exhaustion 
$\{U_k \}_{k \in \mathbb{N} }$ of $M_{\infty}$ and a sequence of smooth diffeomorphisms $\phi_k : U_k \rightarrow V_k \subset M_k$ satisfying: 
\begin{enumerate}
	\item For every compact $K \subset M_{\infty}$, $\phi_k^*F_k |_K$ converges in $C^{\infty}( K, \mathbb{R}^N)$ to $F_{\infty} |_K$
	\item For any compact $A \subset \mathbb{R}^N$ there is some $k_0 \in \mathbb{N}$ such that $(\phi_k^* F_k)(U_k) \cap A = F_k(M_k) \cap A$ for all $k \geq k_0$
\end{enumerate}
We remark that the Langer-Breuning compactness theorem, at least in the form stated above, is not quite satisfactory for our purposes as it does not address the second criterion of the above definition of convergence (which takes care of the second problem alluded to above).
 
\begin{thm}[Compactness theorem for immersed submanifolds]\label{thm: my compactness thm}
Suppose $(F_k , M_k )_{k \in \mathbb{N} }$ is a sequence of proper immersions $F_k : M_k \rightarrow \mathbb{R}^N$ of smooth complete $n$-dimensional manifolds $M_k$ that satisfy the following conditions:
\begin{enumerate}
	\item Uniform curvature derivative bounds: \\
	For each $k \in \mathbb{N}$,  for every $m \in \mathbb{N}$ there exists a constant $C_m(R)$ depending on $m$ and $R$ (and independent of $k$) such that $\abs{ \nabla_k^m h_k}_{ F_k } \leq C_m$
	\item The sequence $( F_k )_{k \in \mathbb{N} }$ does not disappear at infinity: \\
	There exists a radius $R > 0$ such that $B_R(0) \cap F_k(M_k) \neq \emptyset$
for all $k \in \mathbb{N}$. 
\item Local area bound: \\
	For every $R > 0$ there exists a constant $C_R$ depending on $R$ (and independent of $k$), such that
	\begin{equation*} \int_{F^{-1}_k(B_R)} d\mu_g^k \leq C_R. \end{equation*}

\end{enumerate} 
Then there exists a subsequence of $(F_k, M_k)_{ k \in \mathbb{N} }$ which converges on compact sets of $\mathbb{R}^N \times \mathbb{R}$ to a complete proper immersion $(M_{\infty}, F_{\infty})$ that satisfies the same local area bound.
\end{thm}

\begin{proof}
The first step is to obtain a suitable sequence of pointed Riemannian manifolds \\ $\{ (M_k, g_k, p_k) \}_{k \in \mathbb{N} }$ in order to invoke the Cheeger-Gromov compactness theorem.The second assumption of the theorem guarantees that there is at least one sequence of points $(p_k)_{ k \in \mathbb{N} } \in M_k$ whose image lies in some ball of finite radius.  For example, we could take the sequence of points $p_k := \min_{p\in M_k} \abs{ F_k(p) }$; this is possible because $M_k$ is compact for each $k$ and by the second assumption.  As we mentioned in Introduction, a uniform upper bound on the second fundamental form gives a uniform lower injectivity radius bound.  We may invoke the Cheeger-Gromov compactness theorem, and upon passing to a subsequence, we obtain a complete pointed limit manifold $(M_{\infty}, g_{\infty}, p_{\infty})$, an exhaustion $\{U_k \}_{k \in \mathbb{N}}$ of $M_{\infty}$, and a sequence of diffeomorphisms $(\phi_k : U_k \rightarrow V_k \subset M_k )_{k \in \mathbb{N}}$ such that $\phi_k^*g_k$ converges smoothly to $g_{\infty}$ on each compact set $K \subset M$.  By induction, it follows that all higher derivatives of $\phi_k^*F_k$ are uniformly bounded with respect to $g_{\infty}$ on compact sets of $M_{\infty}$.  Passing to a futher subsequence, we obtain smooth convergence of $\phi_k^*F_k$ to a limit immersion $F_{\infty}$ on compact sets of $M_{\infty}$.  At this stage we have shown the first condition in our definition of convergence on compact sets is satisfied.

We now need to show the second condition of our definition is also satisified.  The fact that the area bound holds on the limit is a simple consequence of the $C^1$-convergence of the metrics.  The remaining argument is accomplished by induction and a diagonal sequence argument.  We begin by looking inside a ball $\bar B_1(0)$ in the ambient space.  Suppose that there exists no $k_0 \in \mathbb{N}$ such that $\phi_k^*F_k(U_k) \cap \bar B_1(0) = F_k(M_k) \cap \bar B_1(0)$ for all $k \geq k_0$, then we can pass to a subsequence such that there exists $\tilde{p}_k \in M_k$ such that for all $k$, $\tilde{p}_k \notin V_k$, whilst $F_k(\tilde{p}_k) \in \bar B_1(0)$.  Passing to a further subsequence, we can assume the sequence of pointed manifolds $(M_k, g_k, \tilde{p}_k)$ converges to a limit $(\tilde{M}_{\infty}, \tilde{g}_{\infty}, \tilde{p}_{\infty})$, so that there is an exhaustion $\{ \tilde U_k \}_{k \in \mathbb{N} }$ and diffeomorphisms $\tilde \phi_k : \tilde U_k \rightarrow \tilde V_k \subset M_k$ with $\tilde \phi_k(\tilde p) = (\tilde p_k)$ such that $\tilde \phi_k^*g_k$ converges to $\tilde g$ smoothly on compact sets in $\tilde M$.  As before, by the Arzel\`{a}-Ascoli theorem, passing to another subsequence, we can assume that $\tilde \phi_k^*F_k$ converges smoothly on compact subsets to a limit immersion $\tilde F_{\infty}$.  Now we replace $M_{\infty}$ with $M_{\infty} \sqcup \tilde M_{\infty}$ and repeat the process again.  All of these components intersect with $\bar B_1(0)$, and have area inside $B_2(0)$ bounded below, so by the local area bound this process must stop after finitely may steps, and we have produced a manifold $M$ with finitely many connected components with both parts of the compactness theorem holding on $\bar B_1(0)$.

We complete the proof by induction on the size of the balls in the ambient space:  If we have subsequence for which both parts of the theorem hold on $\bar B_n(0)$, then we add in more components if there are points in $\bar B_{n+1}(0)$ that are in $F_k(M_k)$ but not in $\phi_k^*F_k(U_k)$.  By the same argument, after adding in finitely many components we produce a subsequence and a limit immersion satisfying both parts of the compactness theorem on $\bar B_{n+1}(0)$.
\end{proof}

We can use the above compactness theorem for immersed submanifolds to obtain a compactness theorem for mean curvature flows.  As we mentioned above, our rescaling by remaining time procedure ensures that the second fundamental from remains bounded.  Standard techniques based on the maximum principle can now be used to show that all higher derivative are also uniformly bounded independent of $k$ (see, for example, \cite{Hu2} in the the context of the original continuous rescaling, \cite{Eck} in our context of rescaled flow and also \cite{AB} for further details). The missing essential ingredient is the local area bound, which we shall address in the next section.  We closely follow the exposition in \cite{HA} of Hamilton's proof of his compactness theorem for Ricci flow (\cite{Ham}), adapting it to the mean curvature flow.

\begin{thm}[Compactness theorem for mean curvature flows]\label{thm: compactness for flows}
Suppose that $(F_k , M_k )_{k \in \mathbb{N} }$ is a sequence of proper time-dependent immersions of smooth compact $n$-dimensional manifolds $M_k$ that satisfy the mean curvature flow on the time interval $(a, b)$, where $-\infty \leq a \leq 0 \leq  b \leq \infty$.  Assume that the following conditions are satisifed:
\begin{enumerate}
	\item Uniform curvature derivative bounds: \\
	For each $k \in \mathbb{N}$, there exists a uniform constant $C_0$ such that $\abs{ h_k}_{ F_k } \leq C_0$ on $M_k \times (a,b)$
	\item The sequence doesn't (initially) disappear at infinity: \\
	There exists a time $0$ and radius $R > 0$ such that $B_R(0) \cap F_k(M_k, 0) \neq \emptyset$
for all $k \in \mathbb{N}$. 
\item Initial local area bound: \\
	For every $R > 0$ there exists a constant $C_R$ depending on $R$ (and independent of $k$), such that
	\begin{equation*} \int_{F_k^{-1}(\cdot, \, 0)(B_R(0))} d\mu_{g_{0}}^k \leq C_R. \end{equation*}
\end{enumerate} 
Then there exists a subsequence $(F_k, M_k)_{ k \in \mathbb{N} }$ which converges on compact sets of $\mathbb{R}^N \times \mathbb{R}$ to a complete proper time-dependent immersion $(M_{\infty}, F_{\infty})$ that is also a solution to the mean curvature flow on the time interval $(a,b)$.
\end{thm}
\begin{proof}
The first step in the proof is to produce a sequence of pointed immersions in order to apply the compactness theorem for immersed submanifolds.  We may assume that $0$ in the second condition above is $0$ as $0 \in (a,b)$, then as before we choose a sequence of points $p_k$ such that $\abs{ F_k(p_k, 0) }$ in minimised.  The uniform upper bound on the second fundamental form ensures that $inj_{g_k(0)} \geq \delta > 0$ independent of $k$.  Furthermore, assuming only the bound $\abs{ h }_k \leq C_0$, the interior-in-time higher derivatives estimates (see, for example, \cite{AB}) give bounds on all higher derivatives:
\[ \abs{ \nabla^m h(p,t) }_k \leq C(m,\epsilon, C_0) \]
for all $p \in M$ and $t \in [a + \epsilon, b)$, for each small $\epsilon >0$.  The sequence $(F_k, M_k, g_k(0), p_k)$ satisfies the conditions of the compactness theorem for immersed submanifolds, thus there exists an exhaustion 
$\{U_k \}_{k \in \mathbb{N} }$ of $M_{\infty}$ and a sequence of smooth diffeomorphisms $\phi_k : U_k \rightarrow V_k \subset M_k$ satisfying:
\begin{enumerate}
	\item For every compact $K \subset M_{\infty}$, $\phi_k^*F_k(\cdot,0) |_K$ converges in $C^{\infty}( K, \mathbb{R}^N)$ to $F_{\infty} |_K$; and
	\item For any compact $A \subset \mathbb{R}^N$ there is some $k_0 \in \mathbb{N}$ such that $(\phi_k^* F_k(\cdot,0)(U_k) \cap A = F_k(M_k,0) \cap A$ for all $k \geq k_0$.
\end{enumerate}
Define the diffeomorphisms
\begin{align*}
	\psi_k : U_k \times (a,b) &\rightarrow V_k \times (a,b) \\
	(p,t) &\mapsto (\phi(p), t).
\end{align*}
The idea now is to obtain uniform $C_{\infty}$ control on $\tilde F_k = \psi_k^*F_k$ on compact subsets of $M_{\infty} \times (a,b)$.  Note that the $\tilde F_k$ are defined on $M_k \times (a,b)$, but at the moment $F_{\infty}$ is only defined at $t=0$.  To do this, fix a compact set $Z \subset M_{\infty}$ and $k_0$ sufficiently large so that for all $k \geq k_0$, $Z \subset U_k$.  The first derivative of $F_k$ is the induced metric $g_k$.  The induced metrics $g_k(t)$ are uniformly comparable to $g_{\infty}(0)$ since by the convergence statement they are comparable at $t=0$, and by a standard result (see, for example, \cite{AB}, or \cites{Ham, HA} for the equivalent Ricci flow statement) they remain comparable for $t \in (a,b)$.  To uniformly bound the higher derivatives of $F$, we argue by induction.  In \cite{HA} we derived the following equation for the higher derivatives:
\[ \left| \frac{ \p }{ \p t } \nabla_{\infty}^k \tilde F_k \right|_{g_{\infty}} \leq C(C_0, \epsilon) \left( 1 + \abs{ \nabla_{\infty}^k \tilde F_k}_{g_{\infty}} \right). \]  Arguing by induction (see, for example, \cites{Ham, HA} for the Ricci flow details), we obtain bounds on all higher derivatives independent of $k$ on $Z \times [a+\epsilon, b)$.  Higher derivatives in time are also uniformly bounded, as each derivative in time can be expressed in terms of spatial derivatives using the mean curvature flow equation.  It follows by the Arzel\`{a}-Ascoli theorem (see, for example, \cite{HA}) that there exists a subsequence that converges smoothly on $Z \times [a+\epsilon, b-\epsilon]$.  A diagonal sequence argument then produces a subsequence that converges smoothly on compact sets of $M_{\infty} \times (a,b)$.

The convergence just obtained satisfies the first condition of our definition of convergence on compact sets of $\mathbb{R} \times \mathbb{R}$. We now show that the second condition is also satisfied.  The curvature bounds imply that a point can only move a finite extrinsic distance in finite time: For every $p \in M_k$ such that $F(p, t_1)$ is the ball of radius $R$, then for $a < t_1 \leq t_2 < b$ by the curvature bounds we have
\[ \abs{ F_k(p,t_2) - F_k(p,t_1) } \leq \left| \int_{t_1}^{t_2} \frac{ \p }{ \p t } F_k(p, t) \, dt \right| \leq C\abs{ t_2 - t_1}, \]
so $F_k(p,t_2)$ is in the ball of radius $R + C\abs{ t_2 - t_1 }$ and $p$ is in the image of $\phi_k$ for $k$ sufficiently large.  
\end{proof}
The proof of the theorem adapts in an obvious manner to other geometric flows that satisfy the necessary conditions.  For a similar `compactness theorem for Willmore flows' based on the Breuning-Langer compactness theorem, we direct the reader to \cite{KS}.

\section{Huisken's monotonicity formula and its consequences}
Let $(\bar x, \bar t)$ be a fixed point of $\mathbb{R}^{n+k} \times \mathbb{R}$.  The backwards heat kernel centred at the fixed point $(\bar x, \bar t)$ is defined by
\[ \rho_{\bar x, \bar t} := \frac{ 1 }{ (4\pi(\bar t - t))^{n/2} } e^{ \frac{ \abs{\bar x - x} }{ 4(\bar t - t) } }. \]
When the heat kernel is centred at $(0,0)$ we shall drop the subscripts and simply denote it by $\rho(x,t)$, as opposed to $\rho_{0, 0}(x,t)$.

\begin{thm}[Huisken's monotonicity formula]
Let $F : M^n \times [0, T) \rightarrow \mathbb{R}^{n+k}$ be a solution the mean curvature flow.  Then the monotonicity formula
\[ \frac{d}{dt} \int_M \rho_{\bar x, T}(x,t) \, d\mu_{g(t)} = - \int_M \left| H(x,t) + \frac{1}{2(T-t)} F^{\perp}(x,t) \right|^2 \, d\mu_{g(t)} \]
holds for all $t \in [0, T)$.
\end{thm}
We highlight that Huisken's original proof of the monotonicity formula in \cite{Hu2} is valid in arbitrary codimension.  Note the backwards heat kernel is defined on the ambient space and so we are adhering to the abuse of notation mentioned at the beginning of this chapter.  The centre of our backward heat kernel will most often be $(\hat{p}, T) \in \mathbb{R}^{n+k} \times \mathbb{R}$.  For each pair of times $0 < t_1 < t_2 < T$, the monotonicity formula implies that 
\begin{equation*}
	\int_{ M } \rho_{ \hat{p}, T } \, d\mu_{g_{ t_2 }} \leq \int_{ M } \rho_{ \hat{p}, T } \, d\mu_{g_{ t_1 }}
\end{equation*}
and being the limit of a monotone sequence of decreasing functions, the limit
\begin{equation*}
	\lim_{ t \rightarrow T } \int_{M} \rho_{ \hat{p}, T } \, d\mu_{g_t}
\end{equation*}
certainly exists and is finite.  We shall also use the notation
\begin{equation*}
	\theta(p,t) := \int_{M} \rho_{ \hat{p}, T } \, d\mu_{g_t}
\end{equation*}
and
\begin{equation*}
	\Theta(p) := \lim_{ t \rightarrow T } \theta(p,t).
\end{equation*}
Since $\Theta$ is the limit of a monotone sequence of continuous functions, it follows that $\Theta$ is upper-semicontinuous.  We refer to $\theta$ as the heat density and $\Theta$ as the limit heat density.  An important property of the monotonicity formula is that it is invariant under parabolic rescalings.  By the definition of our parabolic rescaling, for each $k$ we have
\begin{align*}
	\int_{M} \rho_{ \hat{p}, T } \, d\mu_{g_t} &=  \frac{ 1 }{ ( 4\pi(T - t) )^{n/2} } \int_{M} e^{ - \frac{ \abs{x-\hat{p}}^2 }{ 4(T-t) } } \, d\mu_{g_t} \\
	&= \frac{ 1 }{ (-4\pi s)^{n/2} } \int_{ M } e^{ -\frac{ \abs{y}^2}{-4s} } \, d\mu_{g_s}^{(\hat{p}, T), \lambda_k} \\
	&= \int_{ M } \rho \, d\mu_{g_s}^{(\hat{p}, T), \lambda_k}.
\end{align*}
Recalling that $t = T + s/\lambda_k^2$, for each fixed $s \in [-\lambda_k^2 T, 0)$ and all $k$ we have
\[ \int_{M} \rho_{\hat{p},T} \, d\mu_{g_t} =  \int_{M} \rho \, d\mu_{g_{s}}^{(\hat{p}, T), \lambda_{k}}, \]
and consequently
\begin{equation}\label{eqn: rescaled mono 1}
	\lim_{ t \rightarrow T } \int_{M} \rho_{ \hat{p}, T } \, d\mu_{g_t} = \lim_{ k \rightarrow \infty } \int_{M } \rho \, d\mu_{g_s}^k.
\end{equation}
When it is (reasonably) clear which point we are rescaling around, we will often omit the notation $(\hat{p}, T)$ above the measure as we have just done to reduce clutter.  An important application of the monotonicity formula is that it provides the local area bound (independent of k) necessary to apply the compactness theorem for mean curvature flows.  It suffices to obtain the area bound on bounded open subintervals $I_{k_0}:= (-\lambda_{k_0}^2T, 1/k_0)$, as the final argument will be completed by a diagonal sequence argument sending $k_0$ to infinity.  Let us fix a point $p \in M$ and some $k_0$ sufficiently large.  With these choices of $p$ and $k_0$, for all $s \in I_{k_0}$ and every $k \geq k_0$ monotonicity formula gives the estimate
\begin{equation*}
	\int_{M } \rho \, d\mu_ {g_{ s }}^k \leq \int_{ M } \rho_{ \hat{p}, T } d\mu_{ g_{0} } \leq \frac{ \mu_{ g_{0} }(M) }{ (4\pi T)^ \frac{ n }{ 2 } }.
\end{equation*}
We then compute
\begin{align*}
	\int_{F_k^{-1}(B_R)} \, d\mu_{g_s}^k &= \int_{M} \chi_{B_R} \, d\mu_{g_s}^k \\
	&\leq \int_{M} \chi_{B_R} e^{ \frac{ R^2 - \abs{y}^2 }{ -4s } }  \, d\mu_{g_s}^k \\
	&\leq e^{ \frac{ R^2}{ -4s } } \int_{M} e^{ \frac{- \abs{y}^2 }{ -4s } }  \, d\mu_{g_s}^k \\
	&\leq e^{ \frac{ R^2 }{ -4s } } (-4\pi s )^{n/2} \int_{ M } \frac{ 1 }{ (-4\pi s)^{n/2} } e^{ \frac{- \abs{y}^2 }{ -4s } }  \, d\mu_{g_s}^k \\
	&\leq e^{ \frac{ k_0R^2 }{ 4 } } \lambda_{k_0}^n \mu_{g_{0} }(M),
\end{align*}
and thus
\[ \int_{F_k^{-1}(B_R)} \, d\mu_{g_s}^k \leq C_R(M_0, T, k_0). \]

The area bound depends on $k_0$, but not on $k$.  We can now apply Theorem \ref{thm: compactness for flows} to our sequence of rescaled flows $F_k : M \times (-\lambda_{k_0}^2T, 1/{k_0}) \rightarrow \mathbb{R}^{n+k}$ defined by
\[ F_k(p, s) =  \lambda_k \big( F(p, T + s/\lambda_k^2)  - \hat q \big). \] The existence of the limit flow $(F_{\infty}, M_{\infty})$ on the time interval $(-\infty, 0)$ follows by a diagonal sequence argument letting $k_0 \rightarrow \infty$.  We highlight that here $M$ is fixed, and by assumption closed, however $M_{\infty}$ is complete and not necessarily compact.

The area bound can also be obtained without Huisken's monotonicity formula, provided we assume a suitable area growth condition on the initial submanifold.  For a solution of the mean curvature flow (in any codimension) defined on the time interval $t \in (\tau - R^2/(8n), \tau)$, Ecker proves in \cite{Eck}*{Prop. 4.9} the local area decay estimate
\begin{equation}\label{eqn: Ecker}
	\int_{ F^{-1}(B_{R/2}) } d\mu_{g_{\tau}} \leq 8 \int_{ F^{-1}(B_{R}) } d\mu_{g_{\tau - R^2/8n}}.
\end{equation}
Assume now that the initial submanifold satisfies the area growth bound 
\begin{equation}\label{eqn: Ecker1}
	\int_{F^{-1}_0(B_R)} \, d\mu_{g_{0}} \leq AR^n
\end{equation}
for all $R \geq R_0$, where $A$ is a uniform constant and $R_0$ is some fixed large radius.  Combining equations \eqref{eqn: Ecker} and \eqref{eqn: Ecker1} we get
\begin{align}
\int_{ F^{-1}(B_{R}) } d\mu_{g_t} &\leq C(n,T,R_0) \int_{ F^{-1}(B_{2R}) } d\mu_{g_{0}} \\
&\leq C(p,n,T,R_0,A)R^n
\end{align}
which holds for all $t \in [0, T)$ and each $R \geq R_1 > R_0$, where $R_1$ is some new fixed large radius.  Upon rescaling this last estimate we obtain
\begin{equation}\label{eqn: Ecker3}
	\int_{ F_k^{-1}(B_{R}) } d\mu_{g_s}^k \leq \int_{ F_k^{-1}(B_{\lambda_kR}) } d\mu_{g_s}^k \leq C(p,n,T,R_0,A)R^n,
\end{equation}
which is the desired local area bound independent of $k$.

Another consequence of the monotonicty formula is the following important result, which enables us to pass the limit through the integral in the rescaled heat densities.  The result is due independently to Ilmanen \cite{Il} and Stone \cite{St}, who proved it slightly different contexts. Ilmanen proved it in the setting of Brakke flows, while Stone proved it in the context of Huisken's original continuous rescaling argument.  We recast their proof in our setting.

\begin{prop}\label{p: tightness of measures}
Let $F_k : M \times [-\lambda_k^2T, 0) \rightarrow \mathbb{R}^{N}$ be a sequence of proper mean curvature flows of a closed manifold $M$ that subconverges on compact sets of $\mathbb{R}^N \times \mathbb{R}$ to a proper mean curvature flow $F_{\infty} : M_{\infty} \times (\infty, 0) \rightarrow \mathbb{R}^{N}$, where $M_{\infty}$ is a complete manifold.  Assume that for all $R > 0$ the initial submanifold satisfies the area bound
\[ \int_{F^{-1}_0(B_R)} \, d\mu_{g_{0}} \leq AR^n. \]
Then for any given $\epsilon > 0$ and any fixed point $p \in M$, there exists a sufficiently large radius radius $R_2$ such that for each fixed $s \in (-\lambda_{k_0}^2 T, 1/k_0)$ and all $k > k_0$ we have
\begin{equation*}
	\int_{ M \setminus F^{-1}_k(B_{R}) } \rho \, d\mu_{g_s}^{(\hat{p}, T), \lambda_k} \leq \epsilon.
\end{equation*}
\end{prop}
\begin{proof}
By equation \eqref{eqn: Ecker3} above, for all $k\geq k_0$ sufficiently large and each fixed $s \in (-\lambda_{k_0}^2T, 1/k_0)$, we have
\[ \int_{F^{-1}_k(B_R)} d\mu_{g_s}^k \leq CR^n \]
which holds for every $R > R_1$.  For each fixed $s \in ( -\lambda_{k_0}^2 T, 1/k_0)$ we estimate
\begin{align*}
	\int_{ M  \setminus F_k^{-1}(B_{R}) } \rho \, d\mu_{g_s}^{ (\hat{p}, T), \lambda_k} &\leq \frac{ C }{ (-s)^{ \frac{ n }{ 2 } } } \sum_{j=1}^{\infty} \int_{ F_k^{-1} (B_{ R^{j+1} } \setminus B_{R^j} ) } e^{ -R^{2j} / (-4s) } \, d\mu_{g_s}^{ (\hat{p}, T), \lambda_k } \\
	&\leq \frac{ C }{ (-s)^{ \frac{ n }{ 2 } } } \sum_{j=1}^{\infty} R^{ n(j+1) } e^{ -R^{2j}/(-4s) }.
\end{align*}
For each fixed $s \in (-\lambda_{k_0}^2 T, 1/k_0)$, the term on the right can be made as small as we like by choosing $R$ sufficiently large, so for any given $\epsilon$ we can fix $R$ sufficiently large such that the desired estimate holds for all $R \geq R_2$.
\end{proof}

The proposition is, by definition, the statement that the family of weighted measures $\rho \, d\mu_{g_s}^k$ is tight for each fixed $s$.  By Prohorov's Theorem we immediately obtain the following important corollary:
\begin{equation*}
	\lim_{k \rightarrow \infty} \int_{M } \rho \, d\mu_{g_s}^k = \int_{ M_{\infty} } \rho \, d\mu_{g_s}^{\lambda_{\infty} } < \infty.
\end{equation*}
Let us dwell for a second on why this result is important:  The limit manifold $M_{\infty}$ we obtain from the compactness theorem is complete, and not necessarily compact.  Certainly if $M_{\infty}$ contains a compact component, then this component is diffeomorphic to $M$ by definition of the convergence.  However, if $M_{\infty}$ is only complete, as it often will be, then the integral
\[ \int_{M_{\infty}} \rho \, d\mu_{g_s}^{\lambda_{\infty}} \]
could possibly be infinite.  The fact that the weighted family of measure is tight ensures that the measure `does not escape to infinity' in the limit.  We remark that the $C^1$-convergence of $F_k$ and $g_k$ obtained in the compactness theorem can be used to show that $\mu_k \rightarrow \mu$, that is the pushforward measures converge weak-${*}$ in $\mathbb{R}^{n+k}$.

\section{A partial classification of special type I singularities}
In order to probe the shape of the evolving submanifold as the first  singular time is approached, we want to rescale the monotonicity formula around the singular point $\hat{p}$.  A point $p \in M$ is called a general singular point if there exists a sequence of points $p_k \rightarrow p$ and times $t_k \rightarrow T$ such that for some constant $\delta > 0$,
\begin{equation*}
	\abs{ h }^2(p_k, t_k) \geq \frac{ \delta }{ T - t_k }.
\end{equation*}
A point $p \in M$ is called a special singular point if there exists a sequence times $t_k \rightarrow T$ such that for some constant $\delta > 0$,
\begin{equation*}
	\abs{ h }^2(p, t_k) \geq \frac{ \delta }{ T - t_k }.
\end{equation*}
This distinction between singular points is not made in \cite{Hu1}, and the points studied in \cite{Hu1} are actually special singular points (see Definition 2.1 of \cite{Hu1}).  The analysis to cope with moving points was subsequently contributed by Stone in \cite{St}.  We now give a partial classification of special type I singularities in high codimension.

\begin{thm}
Let $M : \times [0, T) \rightarrow \mathbb{R}^{n+k}$ be a solution of the mean curvature flow.  If the evolving submanifold develops a special type I singularity as $t \rightarrow T$, then there exists a sequence of rescaled flows $F_k : M \times [\lambda_k^2 T, 0) \rightarrow \mathbb{R}^{n+k}$ that subconverges to a limit flow $F_{\infty} : M_{\infty} \times (-\infty, 0) \rightarrow \mathbb{R}^{n+k}$ on compact set of $\mathbb{R}^{n+k} \times \mathbb{R}$ as $k \rightarrow \infty$.  Moreover, $F_{\infty} : M_{\infty} \times (-\infty, 0) \rightarrow \mathbb{R}^{n+k}$ satisfies $ H_{\infty} = - 1/(2s)F_{\infty}^{\bot}$ and is not a plane.
\end{thm}
\begin{proof}
The existence of the limit flow is guaranteed, subject to the necessary assumptions, by the compactness theorem for mean curvature flows.  It remains to show the last two assertions of the proposition.  Suppose that the special type I singulariy forms at some point $(\hat{p}, T) \in \mathbb{R}^{n+k} \times \mathbb{R}$, so by definition there exists a sequence times $t_k \rightarrow T$ such that for some constant $\delta > 0$, we have $\abs{ h }^2(p, t_k) \geq \frac{ \delta }{ T - t_k }$.
Rescaling Huisken's monontonicity formula at each scale $\lambda_k = 1/\sqrt{2(T-t_k)}$ about the single fixed point $\hat{p}$ gives
\begin{equation*}
	\frac{ d }{ ds } \int_{ M } \rho \, d\mu_{g_s}^k = - \int_{ M } \rho \Big| H_k + \frac{ 1 }{ 2s } F_k^{ \bot } \Big|^2 \, d\mu_{g_s}^k,
\end{equation*}
which holds for all $k$ and $s \in [-\lambda_k^2T, 0)$.  For any fixed $s_0 \in [-\lambda_k^2T, 0)$ and $\sigma > 0$ we integrate this from $s_0 -  \sigma$ to $s_0$ and rearrange a little to get
\begin{equation*}
			 \int_{s_0 - \sigma }^{ s_0 } \int_{ M } \rho \Big| H_k + \frac{ F_k^{ \bot } }{ 2s } \Big|^2 \, d\mu_{g_s}^k = \int_{ M } \rho \, d\mu_{g_{ s_0 - \sigma } }^k - \int_{ M } \rho \, d\mu_{g_{ s_0 }}^k.
\end{equation*}
We take the limit as $k \rightarrow \infty$, and by equation \eqref{eqn: rescaled mono 1} and Proposition \ref{p: tightness of measures} we have
\begin{equation*}
	\int_{ M_{\infty} } \rho \, d\mu_{g_{ s_0 - \sigma } }^{ \infty } = \lim_{ t \rightarrow T } \int_{M} \rho_{ ( \hat{p}, T ) } \, d\mu_{g_t} = \int_{ M_{\infty} } \rho \, d\mu_{g_{s_0} }^{ \infty }  < \infty.
\end{equation*}
We then conclude, using Proposition \ref{p: tightness of measures} again, that
\begin{align*}
	\lim_{ k \rightarrow \infty } \int_{s_0 - \sigma }^{ s_0 } \int_{ M } \rho \Big| H_k + \frac{ F_k^{ \bot } }{ 2s } \Big|^2 \, d\mu_{g_s}^k &= \int_{ s_0 - \sigma }^{ s_0 } \int_{ M_{\infty} } \rho \Big| H_{ \infty } +\frac{ 1 }{ 2s } F_{ \infty }^{ \bot } \Big|^2 \, d\mu_{g_s}^{ \infty } \\
	&= 0,
\end{align*}
and therefore $ H_{ \infty } = - 1/(2s) F_{ \infty }^{\bot}$ on $s \in [s_0 - \sigma, s_0]$.  Finally, for every scale $\lambda_k$, at the fixed point $p$ at time $s_k = -1$ the rescaled second fundamental form satifies the lower bound
\begin{align*}
	\abs{ h }_k^2(p,s_k) &= \frac{ \abs{ h }^2(p, t_k) }{ \lambda_k^2 } \\
	&\geq 2(T - t_k) \cdot \frac{ \delta }{ -s_k } \cdot \frac1{ 2(T - t_k) } \\
	&= \delta.
\end{align*}
Thus the the limit flow also satisifies $\abs{ h }_{ \infty }^2(p, -1 ) \geq \delta$ and consequently it is not flat.
\end{proof}

We have just shown that the blow-up limit of a type I singularity is self-similar. In order to give a partial classification of these solutions, in addition to assuming that $M_0$ satisifes $\abs{ H }_{ \text{min} } > 0$, we also assume it satisfies the pinching condition $\abs{h}^2 \leq 4/(3n) \abs{ H }^2$.  The pinching condition allows us to eventually reduce the problem to that of classifying hypersurfaces of a $\mathbb{R}^{n+1}$.  The classification result we shall need is the following:

\begin{prop}\label{prop: covariant const class}
Let $F : M^n \rightarrow \mathbb{R}^{n+1}$ be an immersion of a closed manifold.  If $F(M)$ satisfies $\nabla h = 0$, then $F(M)$ is of the form $\mathbb{S}^{p} \times \mathbb{R}^{n-p}$, where $0 \leq p \leq n$.
\end{prop}

A proof of this result can be found in \cite{CdCK} and \cite{Law}.  The result stated in \cite{CdCK} is for hypersurfaces of the sphere, but the adaption to a flat background is straightforward.
	
Let us now commence with classification in the compact case.  Up to now is has been convenient to analyse the limit flow at the time $-1$.  In following it becomes slightly less cumbersome to analyse the limit flow at time $-1/2$, as was done in \cite{Hu2} and \cite{Hu3}.

\begin{thm}
Suppose that $F_{\infty} : M_{\infty}^n \times (-\infty, 0) \rightarrow \mathbb{R}^{n+k}$ arises as a blow-up limit of the mean curvature flow $F: M^n \times [0, T) \rightarrow \mathbb{R}^{n+k}$ about a special singular point.  Assume that $M_0$ satisfies $\abs{ H }_{ \text{min} } > 0$ and $\abs{ h }^2 \leq 4/(3n) \abs{ H }^2$.  If $F_{\infty}(M_{\infty}, -1/2)$ is compact, then it must be a sphere $\mathbb{S}^n(\sqrt{n})$ or one of the cylinders $\mathbb{ S }^{ m }( \sqrt{m} ) \times \mathbb{R}^{n-m}$, where $1 \leq m \leq n-1$.
\end{thm}
\begin{proof}
Take the inner product of $ H = - F^{\bot}$ with $H$ and differentiate to get
\[ 2 \langle \nabla_j  H, H \rangle = -\langle \nabla_j F, H \rangle - \langle F, \nabla_j H \rangle. \]
We use the Gauss relation to compute
\[ \langle \nablap_j H, H \rangle = \big\langle H, \langle F, F_*\p_p \rangle h_{ip} \big\rangle, \]
so $\nablap_j H = \langle F, F_*\p_p \rangle h_{jp}$.  A further differentiation gives
\begin{align}
	\notag \nabla_i \nabla_j H &= \langle F_*\p_i, F_*\p_p \rangle h_{jp} + \langle F, h_{ip} \rangle h_{jp} + \langle F, F_*\p_p \rangle \nabla_p h_{ij} \\
	\notag &=	\langle  F_*\p_i, F_*\p_p \rangle h_{jp} - \langle H, h_{ip} \rangle h_{jp} + \langle F, F_*\p_p \rangle \nabla_p h_{ij} \\
	&= h_{ij} - H \cdot h_{ip} h_{jp} + \langle F, F_*\p_p \rangle \nabla_p h_{ij}. \label{e: compact 2}
\end{align}
Contracting \eqref{e: compact 2} with $g_{ij}$ gives
\begin{equation*}
	\Delta H = H - H \cdot h_{ip} h_{ip} + \langle F, F_*\p_p \rangle \nabla_p H,
\end{equation*}
and after taking the inner product with $H$ we obtain
\begin{equation}\label{e: compact H2}
	\Delta \abs{H}^2 = 2 \abs{H}^2 - 2 \sum_{ i,j } ( H \cdot h_{ij} )^2 + \langle F, F_*\p_p \rangle \nabla_p H \cdot H + 2 \abs{ \nabla H }^2.
\end{equation}
On the other hand, contracting \eqref{e: compact 2} with $g_{ij}$ we get
\begin{equation}\label{e: compact 3}
	h_{ij} \cdot \nabla_i \nabla_j H = \abs{h}^2 - H \cdot h_{ip} h_{ij} \cdot h_{jp} + \langle F, F_*\p_p \rangle \nabla_p h_{ij} \cdot h_{ij}.
\end{equation}
Now recall Simons' indentity: $\Delta \abs{h}^2 = 2 h_{ij} \cdot \nabla_i \nabla_j H + 2 \abs{ \nabla h }^2 + 2Z$.  Combining Simons' indentity and \eqref{e: compact 3} gives
\begin{equation}\label{e: compact h2}
	\Delta \abs{ h }^2 = 2 \abs{ h }^2 + 2\langle F, F_*\p_p \rangle \nabla_p h_{ij} \cdot h_{ij} + 2 \abs{ \nabla h }^2 - 2 \sum_{ \alpha, \beta } \Big( \sum_{ i, j } h_{ij\alpha} h_{ij\beta} \Big) - \abs{ \Rp }^2.
\end{equation}
Note that the term $H \cdot h_{ip} h_{ij} \cdot h_{jp}$ cancels.  The idea now is to examine the scaling-invariant quantitiy $\abs{ h }^2 / \abs{ H }^2$, and to do so, we first establish $\abs{ H } \neq 0$ in order to perform the division.  The strong elliptic minimum principle applied to equation \eqref{e: compact H2} shows that either $\abs{ H } \equiv 0$ or $\abs{ H } > 0$ everywhere.  Since $F_{\infty}(M_{\infty})$ is assumed to be compact, it must be that $\abs{ H } > 0$ everywhere.  Using equations \eqref{e: compact H2} and \eqref{e: compact h2} we compute $\Delta( \abs{ h }^2 / \abs{ H }^2 )$ and obtain
\begin{equation}\label{e: compact h2 / H2}
	\begin{split}
		0 &= \Delta \left( \frac{ \abs{ h }^2 }{ \abs{ H }^2 } \right) - \frac{ 2 }{ \abs{ H }^2 } \big( \abs{ \nabla h }^2 - \frac{ \abs{ h }^2 }{ \abs{ H }^2 } \abs{ \nabla H }^2 \big) + \frac{ 2 }{ \abs{ H }^2 } \big( R_1 - \frac{ \abs{ h }^2 }{ \abs{ H }^2 } R_2 \big) \\		&\qquad + \frac{ 2 }{ \abs{ H }^2 } \nabla_i \abs{ H }^2 \nabla_i \left( \frac{ \abs{ h }^2 }{ \abs{ H }^2 } \right) - \langle F, F_*\p_i \rangle \nabla_i \left( \frac{ \abs{ h }^2 }{ \abs{ H }^2 } \right).
	 \end{split}
\end{equation}
Since $F_{\infty}(M_{\infty}, -1/2)$ is assumed to be compact, the function $\abs{ h }^2 / \abs{ H }^2 $ attains a maximum somewhere in $M$.  At a maximum $\nabla_i ( \abs{ h }^2 / \abs{ H }^2 ) = 0$ and $\Delta ( \abs{ h }^2 / \abs{ H }^2 ) \leq 0$, so at a maximum we have
\begin{equation*}
	0 = \Delta \left( \frac{ \abs{ h }^2 }{ \abs{ H }^2 } \right) - \frac{ 2 }{ \abs{ H }^2 } \big( \abs{ \nabla h }^2 - \frac{ \abs{ h }^2 }{ \abs{ H }^2 } \abs{ \nabla H }^2 \big) + \frac{ 2 }{ \abs{ H }^2 } \big( R_1 - \frac{ \abs{ h }^2 }{ \abs{ H }^2 } R_2 \big).
\end{equation*}
Moreover, from the basic gradient estimate
\[ \abs{\nabla h}^2 \geq \frac{3}{n+2} \abs{\nabla H}^2 \]
and the Pinching Lemma we can estimate
\begin{equation}
	0 \leq \Delta \left( \frac{ \abs{ h }^2 }{ \abs{ H }^2 } \right) - c_1(n) \abs{ \nabla h }^2 - c_2(n) \abs { \ho_1 }^2\abs{ \ho_- }^2 - c_3(n) \abs{ \ho_- }^4,
\end{equation}
where $c_1$, $c_2$ and $c_3$ are positive constants that depend only on $n$. We conclude from the strong elliptic maximum principle that $\abs{ h }^2 / \abs{ H }^2$ must be equal to a constant and $\abs{ \nabla h }^2 = \abs{ \ho_- }^2 = 0$.  This implies that $F_{\infty}(M_{\infty})$ is a hypersurface of some $(n+1)$-subspace of $\mathbb{R}^{n+k}$ with covariant constant second fundamental form, and since was assumed to be compact, from Proposition \ref{prop: covariant const class} it must be a $n$-sphere.
\end{proof}
If $F_{\infty}(M_{\infty}, -1/2)$ is no longer compact then we cannot apply the maximum principle as we have just done.  In this more general case, following \cite{Hu3}, we multiply equation \eqref{e: compact h2 / H2} by $\abs{ h }^2 \rho $ and integrate by parts.  The following theorem includes the previous one as a special case.
\begin{thm}\label{thm: general class}
Suppose that $F_{\infty} : M_{\infty}^n \times (-\infty, 0) \rightarrow \mathbb{R}^{n+k}$ arises as a blow-up limit of the mean curvature flow $F: M^n \times [0, T) \rightarrow \mathbb{R}^{n+k}$ about a special singular point.  If $M_0$ satisfies $\abs{ H }_{ \text{min} } > 0$ and $\abs{ h }^2 \leq 4/(3n) \abs{ H }^2$, then $F_{\infty}(M_{\infty},-1/2)$ must be a sphere $\mathbb{S}^n(\sqrt{n})$ or one of the cylinders $\mathbb{S}^{m}(\sqrt{m}) \times \mathbb{R}^{n-m}$, where $1 \leq m \leq n-1$.
\end{thm}
\begin{proof}
We multiply equation \eqref{e: compact h2 / H2} by $\abs{ h }^2\rho $ and integrate the term involving the Laplacian by parts to achieve
\begin{equation*}
 	\begin{split}
		0 &= -\int_{ M_{\infty} } \Big| \nabla \left( \frac{ \abs{ h }^2 }{ \abs{ H }^2 } \right) \Big|^2 e^{ \frac{ - \abs{ x }^2 }{ 2 } } \, d\mu_{g} - 2 \int_{ M_{\infty} } \frac{ \abs{ h }^2 }{ \abs{ H }^2 } \big( \abs{ \nabla h }^2 - \frac{ \abs{ h }^2 }{ \abs{ H }^2 } \abs{ \nabla H }^2 ) e^{ \frac{ - \abs{ x }^2 }{ 2 } } \, d\mu_{g} \\
		&\qquad + 2 \int_{ M_{\infty} } \frac{ \abs{ h }^2 }{ \abs{ H }^2 } ( R_1 - \frac{ \abs{ h }^2 }{ \abs{ H }^2 } R_2 ) e^{ \frac{ - \abs{ x }^2 }{ 2 } } \, d\mu_{g}. \end{split}
\end{equation*}
The above equation again implies that $\abs{ h }^2 / \abs{ H }^2$ must be equal to a constant and $\abs{ \nabla h }^2 = \abs{ \ho_- }^2 = 0$ and the theorem follows.

\section{General type I singularities}

As we mentioned in the Introduction, because the mean curvature flow in high codimension does not preserve embeddedness we are not able to extend Stone's hypersurface argument to high codimension.  Let us explore a little why this is the case.  Stone's result for hypersurfaces is the following:

\begin{prop}\label{prop: Stone}
Let $F : M^n \times [0, T) \rightarrow \mathbb{R}^{n+1}$ be a solution of the mean curvature flow.  Suppose that $M_0$ is embedded and satisfies $\abs{ H }_{\min} \geq 0$.  If the evolving submanifold develops a type I singularity at some point $p \in M$ as $t \rightarrow T$, then $p$ is a special singular point.
\end{prop}
Stone's analysis shows that it is in fact enough to understand special singular points.  We follow closely \cite{St}, adapting his proof from the continuous rescaling setting to that of rescaled flows.  Stone's argument uses the classification of special type I singularities for hypersurfaces obtained by Huisken in \cite{Hu2} and \cite{Hu3}:

\begin{thm}\label{thm: Huisken classification}
Let $F_{ \infty} (M_{\infty}^n, -1/2) \subset \mathbb{R}^{n+1}$ be a hypersurface that arises as a blow-up limit of the mean curvature flow.  If $M_0$ is embedded and satisfies $H \geq 0$, then $F_{ \infty} (M_{\infty}^n, -1/2)$ must be a hyperplane, the sphere $\mathbb{S}^n(\sqrt{n})$ or one of the cylinders $\mathbb{S}^{m}(\sqrt{m}) \times \mathbb{R}^{n-m}$, where $1 \leq m \leq n-1$.
\end{thm}
Our equivalent theorem for submanifolds is Theorem \ref{thm: general class}.  In Stone's argument, it is essential to conclude that the limit flow is embedded.  We know from Proposition \ref{prop: preservation of embeddedness} that embeddedness of hypersurfaces is preserved.  It is also true that the limit of type I rescalings of embedded hypersurfaces is also embedded; for a proof of this we refer the reader to \cite{Man}.

\begin{proof}[Proof of Proposition \ref{prop: Stone}]
Suppose that $M_t$ is developing a general type I singularity at some point $(p, T)$.  By definition, there exists a sequence of points $p_k \rightarrow p$ and times $t_k \rightarrow T$ such that for some constant $\delta > 0$,
\begin{equation*}
	\abs{ h }^2(p_k, t_k) \geq \frac{ \delta }{ T - t_k }.
\end{equation*}
As before, we want rescale the monotonicity formula, but now we need to rescale about the moving point $\hat{p}_k$.  Rescaling the monotonicity formula about the moving points $\hat{p}_k$ gives
\begin{equation*}
	\frac{ d }{ ds } \int_{ M } \rho \, d\mu_{g_s}^{ (\hat{p}_k, T), \lambda_k} = - \int_{ M } \rho \Big| H_k + \frac{ 1 }{ 2s } F_k^{ \bot } \Big|^2 \, d\mu_{g_s}^{ (\hat{p}_k, T), \lambda_k },
\end{equation*}
which holds for each $k$ and $s \in [-\lambda_k^2T, 0)$.  For any fixed $s_0 \in [-\lambda_k^2T, 0)$ and $\sigma > 0$ we integrate this from $s_0 -  \sigma$ to $s_0$ and rearrange a little to get
\begin{equation}\label{eqn: Stone's Thm 1}
			 \int_{s_0 - \sigma }^{ s_0 } \int_{ M } \rho \Big| H_k + \frac{ F_k^{ \bot } }{ 2s } \Big|^2 \, d\mu_{g_s}^{ (\hat{p}_k, T), \lambda_k } = \int_{ M } \rho \, d\mu_{g_{ s_0 - \sigma } }^{ (\hat{p}_k, T), \lambda_k } - \int_{ M } \rho \, d\mu_{g_{ s_0 } }^{ (\hat{p}_k, T), \lambda_k }.
\end{equation}
The difficulty now is that in general, $\lim_{ k \rightarrow \infty } \theta(p_k, t_k) \neq \Theta(p)$.  The proof is now by contradiction.  If $p$ is a general singular point but is not a special singular point, then by definition there exists some function $\epsilon(t)$ with $\epsilon(t) \rightarrow 0$ as $t \rightarrow T$ such that
\begin{equation*}
	\abs{ h }^2(p,t) \leq \frac{ \epsilon(t) }{ 2( T - t ) }
\end{equation*}
for all time $t \in [0, T)$.  This implies that any blow-up about the single fixed point $\hat{p}$ would satisfy $\abs{ h }^2 = 0$.  From Theorem \ref{thm: Huisken classification} we know that a blow-up around a special singular point is one of $n+1$ different hypersurfaces.  Furthermore, the heat density function evaluated on these hypersurfaces takes on $n+1$ distinct values, of which $1$ is the smallest, which corresponds to a unit multiplicity plane.  Full details of these calculations can be found in the Appendix of \cite{St}.  Crucially, since $M_{ \infty }$ is also embedded, it can only be a unit multiplicity plane, and not a plane of higher mulitplicity.  Since $\Theta$ is upper-semicontinuous, it is actually continuous at $p$, and therefore $\Theta = 1$ in a whole neighbourhood of $p$.  Dini's Theorem on the monotone convergence of functions now implies for $k$ sufficiently large, that $\theta(p_k, t_k) \rightarrow \Theta(p)$ uniformly.  This is the point at which the argument breaks down in high codimenion: since embeddedness of the initial submanifold is not preserved, the blow-up limit may be a plane of higher multiplicity, and thus  $\Theta(p)$ could be any integer.  Therefore, we cannot conclude that $\Theta$ is continuous at $p$, and Dini's Theorem is no longer applicable.

We complete Stone's argument: Returning now to equation \eqref{eqn: Stone's Thm 1}, for every fixed $s_0$ and every fixed point $\hat{p}_k$ the monontonicty formula implies
\begin{equation*}
-\int_{ M } \rho \, d\mu_{g_{ s_0 }}^{ (\hat{p}_k, T), \lambda_k } \leq -\int_{ M } \rho \, d\mu_{g_{s_0}}^{ (\hat{p}_k, T), \lambda_l }
\end{equation*}
for all $l > k$.  Estimating as such, for all $l > k$ we have
\begin{equation*}
			 \int_{s_0 - \sigma }^{ s_0 } \int_{ M } \rho \Big| H_k + \frac{ F_k^{ \bot } }{ 2s } \Big|^2 \, d\mu_{g_s}^{ (\hat{p}_k, T), \lambda_k } \leq \int_{ M } \rho \, d\mu_{g_{ s_0 - \sigma }}^{ (\hat{p}_k, T), \lambda_k } - \int_{ M } \rho \, d\mu_{g_{ s_0 }}^{ (\hat{p}_k, T), \lambda_l }.
\end{equation*}
Sending $l \rightarrow \infty$ and using Proposition \ref{p: tightness of measures} we obtain
\begin{equation*}
	\int_{s_0 - \sigma }^{ s_0 } \int_{ M } \rho \Big| H_k + \frac{ F_k^{ \bot } }{ 2s } \Big|^2 \, d\mu_{g_s}^{ (\hat{p}_k, T), \lambda_k } \leq \int_{ M } \rho \, d\mu_{g_{ s_0 - \sigma } }^{ (\hat{p}_k, T), \lambda_k } - \Theta(p_k).
\end{equation*}
By Dini's Theorem, given any $\epsilon > 0$, there exists a $k_0$ such that for all $k > k_0$ we have
\begin{equation*}
	\int_{s_0 - \sigma }^{ s_0 } \int_{ M } \rho \Big| H_k + \frac{ F_k^{ \bot } }{ 2s } \Big|^2 \, d\mu_{g_s}^{ (\hat{p}_k, T), \lambda_k } \leq \epsilon
\end{equation*}
and thus
\begin{equation*}
	\lim_{ k \rightarrow \infty } \int_{s_0 - \sigma }^{ s_0 } \int_{ M } \rho \Big| H_k + \frac{ F_k^{ \bot } }{ 2s } \Big|^2 \, d\mu_{g_s}^{ (\hat{p}_k, T), \lambda_k } = 0.
\end{equation*}
Using the blow-up procedure of the previous section we obtain a limit flow on $(-\infty, 0)$, and by Proposition \ref{p: tightness of measures} the limit solution satisfies $H_{ \infty } = -1/(2s)F_{ \infty }$ and is again not flat.  This is a contradicton, since by Proposition \ref{p: tightness of measures},

\begin{equation*}
	\lim_{ k \rightarrow \infty } \theta(p_k, t_k) = \Theta(p) = 1,
\end{equation*}
which implies the limit solution is a plane and hence flat.
\end{proof}

\section{Applications of Hamilton's blowup procedure for type I and II singularities}
In this section we present two applications of Hamilton blowups to the mean curvature flow.  In the first case, we show how a type I Hamilton blowup procedure and the compactness theorem for mean curvature flows can be used instead of a normalised flow to give an alternate proof of the limiting spherical shape of the evolving submanifold considered in \cite{Hu1} and \cite{AB}.  The interested reader may like compare the following with the corresponding argument in the Ricci flow, which can found, for example, in \cite{Top}.   Here the Codazzi equation performs the same role as the contracted second Bianchi indentity, and the Codazzi Theorem that of Schur's Theorem.  For a proof of the Codazzi Theorem we refer the reader to \cite[Thm. 26]{Sp}.  Pick any sequence of times $(t_k)_{k \in \mathbb{N}}$ such that $t_k \rightarrow T$ as $k \rightarrow \infty$.  The Pinching Lemma implies that $\abs{ h }^2$ and $\abs{ H }^2$ have equivalent blow-up rates, so we can in fact rescale by $\abs{ H }^2$.  Since $M$ is assumed to be closed, we can pick a sequence of points $(p_k)_{ k \in \mathbb{N} }$ defined by
\begin{equation*}
	\abs{H}(p_k, t_k)  = \max_{ p \in M }\abs{ H }( p, t_k ).
\end{equation*}
For convenience, set $\lambda_k := \abs{H}(p_k, t_k)$.  We define a sequence of rescaled and translated flows by
\begin{equation*}
	F_k(q,s) = \lambda_k \big( F(q, t_k + s/\lambda_k^2) - F(p_k,t_k) \big),
\end{equation*}
where for each $k$, $F_k : M \times [\lambda_k^2T, 0] \rightarrow \mathbb{R}^{n+k}$ is a solution of the mean curvature flow (in the time variable $s$).  The second fundamental form of the rescaled flows is uniformly bounded above independent of $k$ and we can apply the compactness theorem for mean curvature flows to obtain a smooth limit solution of the mean curvature flow $F_{\infty} : M_{\infty} \times (-\infty, 0] \rightarrow \mathbb{R}^{n+k}$.  Futhermore, at $s=0$ the limit solution satisfies $\abs{ H }^2_k = 1$ by construction, so the limit is not flat.  By definition of the rescaling, the second fundamental form rescales as $\abs{ h }_k^2 = \abs{h}^2/ \lambda_k^2$, so the estimate of Theorem 5.1 of \cite{Hu1} or Theorem 4 of \cite{AB} rescales as
\begin{equation*}
	\abs{ \ho }^2_k\leq C_0 \lambda_k^{ -\delta }\abs{ H }^2_k.
\end{equation*}
The limit therefore satisfies
\begin{equation}\label{e: blow up 1}
	\abs{ \ho }^2_{ \infty } = 0,
\end{equation}
which implies that $F_{\infty}(M_{\infty},t)$ is a totally umbilic submanifold.  By the Codazzi Theorem, $F_{\infty}(M_{\infty},t)$ must be plane or a $n$-sphere lying in a $(n+1)$-dimensional affine subspace of $\mathbb{R}^{n+k}$. We know that $\abs{ H }_{ \infty }(\cdot, 0) = 1$ at some point, so $F_{\infty}(M_{\infty},0)$ is not a plane.
\end{proof}

Using Hamilton's blowup procedure for type II singularities and the Pinching Lemma, we can give a rudimentary classification theorem for type II singularities of the mean curvature flow in arbitrary codimension.  A similar classification result for hypersurfaces appears in \cite{HS}.  The proof of the following result follows along the sames lines as \cite{HS}.

\begin{thm}
Let $F : M^n \times [0, T) \rightarrow \mathbb{R}^{n+k}$ be a solution of the mean curvature flow.  Assume that the initial submanifold is closed and satisfies $\abs{ H }_{min} > 0$ and $\abs{ h }^2 < 4/(3n) \abs{ h }^2$.  If the evolving submanifold develops a type II singularity at $t \rightarrow T$, then there exists a sequence of rescaled flows $F_k : M \times I_k \rightarrow \mathbb{R}^{n+k}$ that subconverges to a limit flow $F_{\infty} : M_{\infty} \times (-\infty, \infty) \rightarrow \mathbb{R}^{n+k}$ on compact sets of $\mathbb{R}^{n+k} \times \mathbb{R}$ as $k \rightarrow \infty$.  The limit flow satisfies $0 < \abs{ H }_{ \infty } \leq 1$ and is equal to one at least at one point, and furthermore has positive scalar curvature everywhere.
\end{thm}

\section{A new estimate for the normalised mean curvature flow}
The following estimate simplifies Huisken's original proof in \cite{Hu1} of the exponential convergence of the normalised mean flow to a sphere.  We refer the reader to sections 9 and 10 of \cite{Hu1} for detail concerning the normalised flow.  We avoid the use of integral estimates and Sobolev inequalities and use only the extreme value theorem and the maximum principle.

\begin{prop}
Suppose $\tilde F_{ \tilde{ t } }(M)$ is an initially strictly convex hypersurface immersed in $\mathbb{R}^{n+1}$ moving by the normalised mean curvature flow.  For all time $\tilde{t} \in [0, \infty)$ we have the estimate
\begin{equation*}
	\abs{ \tilde{\nabla} \tilde{h} }^2 + \abs{ \tilde{\ho} }^2 \leq Ce^{ -\delta \tilde{t} }.
\end{equation*}
\end{prop}
\begin{proof}
The idea is to consider $f:= \epsilon \abs{ \nabla h }^2 + N \abs{ \ho }^2/\abs{ H }^2$, where $\epsilon > 0$ will be chosen small and $N$ sufficiently large.  For the moment we work in the un-normalised setting.  The evolution equation for $\abs{ \nabla h }^2$ is of the form
\begin{equation*}
	\frac{ \p }{ \p t } \abs{ \nabla h }^2 = \Delta \abs{ \nabla h }^2 - 2\abs{ \nabla^2 h }^2 + h*h*\nabla h * \nabla h,
\end{equation*}
so we obtain the estimate
\begin{equation*}
	\frac{ \p }{ \p t } \abs{ \nabla h }^2 \leq \Delta \abs{ \nabla h }^2 - 2\abs{ \nabla^2 h }^2 + c_1 \abs{ H }^2 \abs{ \nabla h }^2.
\end{equation*}
The evolution equation for $\abs{ \ho }^2 / \abs{ H }^2$ is given by
\begin{equation*}
	\frac{ \p }{\p t} \left( \frac{ \abs{ \ho }^2 }{ \abs{ H }^2 } \right) = \Delta \left( \frac{ \abs{ \ho }^2 }{ \abs{ H }^2 } \right) + \frac{ 2 }{ \abs{ H }^2 } \big\langle \nabla_i \abs{ H }^2, \nabla_i \left( \frac{ \abs{ \ho }^2 }{ \abs{ H }^2 } \right) \big\rangle - \frac{ 2 }{ \abs{ H }^2 } \abs{ H \cdot \nabla_i h_{kl} - \nabla_i H \cdot h_{kl} }^2
\end{equation*}
(see Lemma 5.2 of \cite{Hu1} and set $\sigma = 0$).
The importance of including the gradient term $\abs{ \nabla h }^2$ in $f$ is the following: the antisymmetric part of $\abs{ \nabla h }^2$ contains curvature terms which we can use to obtain exponential convergence.  We split $\nabla^2 h$ into symmetric and anti-symmetric components, and upon discarding the the symmetric part we obtain
\begin{align*}
	\abs{ \nabla^2 h }^2 &\geq \frac{ 1 }{ 4 } \abs{ \nabla_i \nabla_j h_{kl} - \nabla_k \nabla_l h_{ij} }^2 \\
	&= \frac{ 1 }{ 4 } \abs{ R_{ikjp}h_{pl} + R_{iklp}h_{jp} }^2,
\end{align*}
where the last line follows from Simons' identity.  Some computation shows
\begin{equation*}
	\abs{ R_{ikjp}h_{pl} + R_{iklp}h_{jp} }^2 = 4 \sum_{i, j} ( \kappa_i^2 \kappa_j^4 - \kappa_i^3 \kappa_j^3 ),
\end{equation*}
then using that $\kappa_{ \text{min} } > 0$ we estimate
\begin{align}
	\notag \sum_{i, j} ( \kappa_i^2 \kappa_j^4 - \kappa_i^3 \kappa_j^3 ) &\geq \kappa_{ \text{min} }^2 \sum_{ i < j } \kappa_i \kappa_j(\kappa_i - \kappa_j)^2 \\
	&\geq n\kappa_{ \text{min} }^4 \abs{ \ho }^2 := \epsilon_1 \abs{ \ho }^2. \label{eqn: new max arg 1}
\end{align}
The next important step is to estimate the term $ \abs{ H \cdot \nabla_i h_{kl} - \nabla_i H \cdot h_{kl} }^2$ from below in terms of $\abs{ \nabla h }^2$.  It is a relatively simple matter to estimate this term from below in terms of $\abs{ \nabla H }^2$, however we want to use this good negative term to control the bad reaction term $c_1\abs{ H }^2 \abs{ \nabla h }^2$ of the evolution equation for $\abs{ \nabla h }^2$, so we need an estimate in terms of $\nabla h$.  To do this, as always let $h$ denote the second fundamental form and $B$ a totally symmetric three tensor (we have $\nabla h$ in mind).  Consider the space $\mathcal{A} := \{ h, B : \abs{ h }^2 =1, \abs{ B }^2 =1 \}$, and we also assume strict convexity of $h$.  The conditions on $h$ and $B$ imply this space is compact.  Now consider the function $G(B) =  \abs{ h_{pp} \cdot B_{ijk} - h_{ij} B_{kpp} }^2$.  We claim $G(B) \geq \delta$ for some $\delta > 0$.  Since $\mathcal{A}$ is compact, by the extreme value theorem $G$ assumes its minimum value at some element of $\mathcal{A}$.  We show by contradiction that $G \neq 0$ which proves the claim.  The anti-symmetric part of $G$ is $\abs{ B_{ipp} \cdot h_{jk} - B_{jpp} \cdot h_{ik} }^2$.  We compute at a point where $G$ obtains its minimum, and rotating coordinates so that $e_1 = \nabla H / \abs{ \nabla H }$ we have
\begin{equation*}
	\abs{ B_{ipp} \cdot h_{jk} - B_{jpp} \cdot h_{ik} }^2 = \abs{ B_{ipp} }^2 \left( \abs{ h }^2 - \sum_{ k =1 }^n h_{1k}^2 \right),
\end{equation*}
so if $G=0$ then $\abs{ B_{ipp} }^2 = 0$ or $ \abs{ h }^2 = \sum_{ k =1 }^n h_{1k}^2$.  The latter implies that $\abs{ h }^2 = h_{11}^2$, which contradicts the strict convexity of the hypersurface.  Therefore, if $G(B) = 0$, then $\abs{ B_{ipp} }^2 = 0$.  From the definition of $G$ it now follows that the full tensor $\abs{ B }^2 = 0$.  This contradicts $\abs{ B }^2 =1$ and the claim follows.  The term $h_{pp} \cdot B_{ijk} - h_{ij} B_{kpp}$ is a quadratic form, so for arbitrary $h$ and $B$ we obtain $G(B) \geq \delta \abs{ h }^2 \abs{ B }^2$ by scaling.  Applying this to our situation, we have $\abs{ B }^2 = \abs{ \nabla h }^2$, then estimating $\abs{ h }^2 \geq n \kappa_{\text{min}}^2$ we obtain
\begin{equation}
	\abs{ H \cdot \nabla_i h_{kl} - \nabla_i H \cdot h_{kl} }^2 \geq \delta n \kappa_{\text{min}}^2 \abs{ \nabla h }^2 := \epsilon_2 \abs{ \nabla h }^2 . \label{eqn: new max arg 2}
\end{equation}
Returning now to the evolution equation for $f$, converting to the normalised setting and using the estimates \eqref{eqn: new max arg 1} and \eqref{eqn: new max arg 2} we get
\begin{align*}
	\frac{ \p }{ \p \tilde{ t } } \tilde{ f }  &\leq \tilde{ \Delta } \tilde{ f } - \epsilon_1 \abs{ \tilde{ \ho } }^2 + c_1 \abs{ \tilde{ H } }^2 \abs{ \tilde{ \nabla } \tilde{ h } }^2 + \frac{ 2 }{ \abs{ \tilde{ H } }^2 }  \big\langle \tilde{ \nabla }_i \abs{ \tilde{ H } }^2, \tilde{ \nabla }_i \tilde{ f } \big\rangle - \frac{ 2 }{ \abs{ \tilde{ H } }^2 } \big\langle \tilde{ \nabla }_i \abs{ \tilde{ H } }^2, \tilde{ \nabla }_i ( \epsilon \abs{ \tilde{ \nabla } \tilde{ h } }^2 ) \big\rangle \\
	&\qquad - \frac{ 2 \epsilon_2 N }{ \abs{ \tilde{ H } }^2 } \abs{ \tilde{ \nabla } \tilde{ h } }^2 - \frac{ 4 \epsilon }{ n } \tilde{ \hbar } \abs{ \tilde{ \nabla } \tilde{ h } }^2.
\end{align*}
In the normalised setting the second fundamental form, and therefore all higher derivatives, are bounded above.  We can therefore estimate
\begin{equation*}
	\big\langle \tilde{ \nabla }_i \abs{ \tilde{ H } }^2, \tilde{ \nabla }_i ( \epsilon \abs{ \tilde{ \nabla } \tilde{ h } }^2 ) \big\rangle \leq 4 \abs{ \tilde{ H } } \abs{ \tilde{ \nabla } \tilde{ H } } \abs{ \tilde{ \nabla } \tilde{ h } } \abs{ \tilde{ \nabla }^2 \tilde{ h } } \leq C \abs{ \tilde{ \nabla } \tilde{ h } }^2.
\end{equation*}
Using $0 < C_{ \text{min} } \leq \abs{ \tilde{ H } }_{ \text{min} } \leq \abs{ \tilde{ H } }_{ \text{max} } \leq  C_{ \text{max} }$, we make $N$ sufficiently large to consume the bad $\abs{ \tilde{ \nabla } \tilde{ h } }^2$ terms and then we discard these terms.  Using again $C_{ \text{min} } \leq \abs{ \tilde{ H } }_{ \text{min} }$ we estimate
\begin{equation*}
		-\epsilon_1 \abs{ \tilde{ \ho } }^2 - \frac{ 4 \epsilon }{ n } \tilde{ \hbar } \abs{ \tilde{ \nabla } \tilde{ h } }^2 \leq - \delta \tilde{ f }
\end{equation*}
for some small $\delta$.  We ultimately obtain  
\begin{equation*}
	\frac{ \p }{ \p \tilde{ t } } \tilde{ f }  \leq \tilde{ \Delta } \tilde{ f } + \tilde{U}^k \tilde{ \nabla }_k \tilde{ f } - \delta \tilde{ f }.
\end{equation*}
This implies
\begin{equation*}
	\frac{ \p }{ \p \tilde{ t } } ( e^{ \delta \tilde{ t } } \tilde{ f } ) \leq \tilde{ \Delta } ( e^{ \delta \tilde{ t } } \tilde{ f } ) + U^k \tilde{ \nabla }_k ( e^{ \delta \tilde{ t } } \tilde{ f } ),
\end{equation*}
and from the maximum principle we conclude $e^{ \delta \tilde{ t } } \tilde{ f } \leq C$ and the theorem follows since $\abs{ \tilde{ H } }_{ \text{max} } \leq  C_{ \text{max} }$.
\end{proof}
Note that we obtain exponential decay of both $\abs{ \tilde{ \nabla } \tilde{ h } }^2$ and $\abs{ \tilde{ \ho } }^2$ at the same time.  Since we have pointwise control on the decay of $\abs{ \tilde{ \ho } }^2$, exponential decay of the higher derivatives can be proved by the maximum principle in a similar manner as the un-normalised estimates.  The important modification needed is that one adds in $\abs{ \tilde{ \ho } }^2$, which is exponentially decaying, rather than $\abs{ h }^2$, which is not, to generate the favourable gradient terms.  The same proof goes through in the high codimension setting of \cite{AB} provided we can estimate a lower bound for
\begin{align*}
	\abs{ \nabla^2 h }^2 &\geq \frac{ 1 }{ 4 } \abs{ \nabla_i \nabla_j h_{kl} - \nabla_k \nabla_l h_{ij} }^2 \\
	&= \frac{ 1 }{ 4 } \abs{ \Rp_{ik\alpha\beta}h_{jl\alpha}\nu_{\beta} + R_{ikjp}h_{pl} + R_{iklp}h_{jp} }^2
\end{align*}
in terms of $\abs{ \ho }^2$.  Such an estimate could hold for $c < 1/(n-1)$, although we have not seriously attempted to do this calculation.  We remark that it would be nice to use the same idea in the un-normalised setting, and avoid the integral estimates completely.  Unfortunately, at the moment we can only make such an argument work if the submanifold is already extremely pinched.

\end{document}